\documentclass[12pt]{amsart}
\usepackage{latexsym,amssymb,amsmath,tikz, xypic}

\numberwithin{equation}{section}
\newtheorem{theorem}{Theorem}[section]
\newtheorem{lemma}[theorem]{Lemma}

\theoremstyle{definition}
 
\newtheorem{remark}[theorem]{Remark}
\newtheorem{example}[theorem]{Example}

\begin{document}

 
\title{Regularity and $h$-polynomials of edge ideals}
\thanks{Submitted Version: October 16, 2018}

\author[T. Hibi]{Takayuki Hibi}
\address{Department of Pure and Applied Mathematics, Graduate School
of Information Science and Technology, Osaka University, Suita, Osaka
565-0871, Japan}
\email{hibi@math.sci.osaka-u.ac.jp}

\author[K. Matsuda]{Kazunori Matsuda}
\address{Kitami Institute of Technology, Kitami, Hokkaido 090-8507, Japan}
\email{kaz-matsuda@mail.kitami-it.ac.jp}
 
\author[A. Van Tuyl]{Adam Van Tuyl}
\address{Department of Mathematics and Statistics\\
McMaster University, Hamilton, ON, L8S 4L8, Canada}
\email{vantuyl@math.mcmaster.ca}

 
\begin{abstract}
For any two integers $d,r \geq 1$, we show that there exists
an edge ideal $I(G)$ such that the ${\rm reg}\left(R/I(G)\right)$, 
the Castelnuovo-Mumford regularity
of $R/I(G)$, is $r$, and $\deg h_{R/I(G)}(t)$, the degree of the $h$-polynomial 
of $R/I(G)$, is $d$.   Additionally, if $G$ is a graph on $n$ vertices,
we show that ${\rm reg}\left(R/I(G)\right) + \deg h_{R/I(G)}(t) \leq n$.
\end{abstract}
 
\maketitle


\section{Introduction}

Let $I$ be a homogeneous ideal of the polynomial ring $R = k[x_1,\ldots,x_n]$
where $k$ is a field.  Associated to $I$ is a graded minimal free
resolution of the form
\[0 \rightarrow \bigoplus_{j \in \mathbb{N}} R(-j)^{\beta_{p,j}(I)}
\rightarrow \cdots \rightarrow \bigoplus_{j \in \mathbb{N}} R(-j)^{\beta_{1,j}(I)}
\rightarrow R \rightarrow R/I \rightarrow 0\]
where $R(-j)$ denotes the polynomial ring $R$ with its grading twisted by
$j$, and $\beta_{i,j}(I)$ is the ${i,j}$-th graded Betti number.  This
resolution encodes a number of important invariants of $R/I$.  One
such invariant is the {\it (Castelnuovo-Mumford) regularity} of $I$,
which is defined by
\[{\rm reg}(R/I) = \max\{j-i ~|~ \beta_{i,j}(I) \neq 0\}.\]
The {\it Hilbert series} of $R/I$, that is, 
$H_{R/I}(t) = \sum_{j \in \mathbb{N}} \dim_k (R/I)_jt^j$,  can also be read from this resolution;
in particular, if $b_{i,i+j} = \beta_{i,i+j}(I)$, then (see \cite[p. 100]{hhGTM})
\[
H_{R/I}(t) = \frac{\sum_{i}(-1)^i\left(\sum_j b_{i,i+j}t^{i+j}\right)}{(1-t)^n}.
\]
This rational function may or may not be in lowest terms; when we rewrite
$H_{R/I}(t)$ in lowest terms, the Hilbert-Serre theorem (see \cite[Proposition 4.4.1]{BH}) says
\[H_{R/I}(t) = \frac{h_{R/I}(t)}{(1-t)^{\dim(R/I)}} ~~\mbox{with $h(t) \in \mathbb{Z}[t]$
and $h(1) \neq 0$}.\]
The polynomial $h_{R/I}$ is called the {\it $h$-polynomial} of $R/I$.

Given that  ${\rm reg}(R/I)$ and $\deg h_{R/I}(t)$ are both derived from
the graded minimal free resolution, 
one can ask if there is any relationship between
these two invariants.  For example, from \cite[Lemma 4.1.3]{BH}, it
follows that if $I$ has a {\it pure resolution} (for each $i$,
there is at most one $j$ such that $\beta_{i,i+j}(I) \neq 0$), then 
\[\deg h_{R/I}(t) - {\rm reg}(R/I) = \dim (R/I) - {\rm depth} (R/I).\] 
The first two authors initiated a comparison of these two invariants
in \cite{HM1,HM2,HM3}.  It was shown in \cite{HM1} that for all $r,d \geq 1$,
there exists a monomial ideal such that ${\rm reg}(R/I) =r $ and 
${\rm deg}(R/I) = d$; in \cite{HM2}, it shown that this monomial ideal
could be taken to be a lexsegment monomial ideal.  In both cases,
the degrees of the minimal generators of $I$ depend upon on $r$ and/or $d$.    
However, if restrict our family of ideals, one might expect some 
restriction on the values of $r$ and $d$.  For example, it is shown in
\cite{HM3} that for $2 \leq r \leq d$, there exists a binomial edge 
ideal (see \cite{HHHKR,O}) $J_{G}$ with ${\rm reg}(R/J_{G}) =r$ and $\deg h_{R/J_{G}}(t)=d$, and furthermore,
\cite[Theorem 2.1]{SK} says that $\deg h_{R/J_{G}}(t)=1$ if ${\rm reg}(R/J_{G}) =1$. 

The starting point of this paper is to ask what happens if we restrict
to edge ideals.  Recall that if $G = (V(G),E(G))$ is a finite simple graph
on $V(G) = \{x_1,\ldots,x_n\}$, then the {\it edge ideal} is the
ideal $I(G) = ( x_ix_j ~|~ \{x_i,x_j\} \in E ) \subseteq 
R = k[x_1,\ldots,x_n]$.  Our main
result is the perhaps surprising fact that one can obtain the main
result of \cite{HM1} using only edge ideals (unlike \cite{HM1,HM2}
where the degrees of the generators change, our generators always have degree
two):

\begin{theorem}[Theorem \ref{Main1}] \label{maintheorem}
Let $r, d \geq 1$ be integers.    
Then there is a finite simple graph
$G$ with $r = {\rm reg}\left(R/I(G)\right)$ and $d = \deg h_{R/I(G)}(t)$.
\end{theorem}

\noindent
Our proof of Theorem \ref{maintheorem} uses the following strategy.
We show that if $G$ is a graph with ${\rm reg}\left(R/I(G)\right) = r$ and
$\deg h_{R/I(G)}(t) = d$, then one can construct a new graph $G'$
with ${\rm reg}(R/I(G')) = r+1$ and
$\deg h_{R/I(G')}(t) = d+1$.  The proof of Theorem \ref{maintheorem}
then reduces to constructing graphs with $({\rm reg}\left(R/I(G)\right),\deg h_{R/I(G)}(t))
= (1,d)$ or $(r,1)$ for any integers $d,r \geq 1$.

Interestingly, ${\rm reg}\left(R/I(G)\right)$ and $\deg h_{R/I(G)}(t)$ are related
by the following inequality.
\begin{theorem}[Theorem \ref{maintheorem2}] \label{secondresult}
Let $G$ be a graph on $n$ vertices.  Then
\[{\rm reg}\left(R/I(G)\right) + \deg h_{R/I(G)}(t) \leq n.\]
\end{theorem}

\noindent
We provide examples to show that this bound is sharp.  Note that
Theorem \ref{secondresult} gives a new upper bound on the regularity
of edge ideals, i.e.,
${\rm reg}\left(R/I(G)\right) \leq n - \deg h_{R/I(G)}(t)$, which complements
past research on the regularity of edge ideals (see \cite{Ha,HVT}).

\noindent
{\bf Acknowledgments.} 
The first and last author began discussions on this project at the 
BIRS (Banff International Research Station) workshop entitled
{\it New Trends in Syzygies}, organized by 
Giulio Caviglia and Jason McCullough and held in June 2018.  We thank the organizers and 
BIRS for providing a stimulating research environment. 
Experiments with {\it Macaulay2} \cite{GS} led to many of our results. 
Hibi and Matsuda's research was supported by JSPS KAKENHI 
26220701 and 17K14165. 
Van Tuyl's research was supported by NSERC Discovery Grant 2014-03898.
This work was also made possible by the facilities of the Shared Hierarchical 
Academic Research Computing Network (SHARCNET: {\tt www.sharcnet.ca}) and 
Compute/Calcul Canada.


\section{Background}\label{prelim-sec}

We recall the relevant graph theory and commutative algebra background.
We continue to use the notation and terminology from the introduction.

Let $G = (V(G),E(G))$ be a finite simple graph on the vertex
set $V(G)  = \{x_1,\ldots,x_n\}$ and edge set $E(G)$ consisting of unordered
pairs of distinct elements of $V(G)$, that is, if $e \in E(G)$, then
$e = \{x_i,x_j\}$ for some $i\neq j$.  If $G$ is clear, we 
write $V$, respectively $E$, for $V(G)$, respectively $E(G)$.

 We say that there is
a {\it path} between the vertices $x_i$ and $x_j$ if there is
a collection of edges $\{e_1,e_2,\ldots,e_t\}$ such that $x_i \in e_1,
x_j \in e_t$, and $e_\ell \cap e_{\ell+1} \neq \emptyset$ for all 
$\ell=1,\ldots,r-1$.
A graph $G$ is {\it connected} if there is a path between
every pair of vertices of $G$; otherwise, $G$ is said to be
{\it disconnected}.   A {\it connected component} of $G$ is 
a maximal connected subgraph.

Given any subset $W \subseteq V(G)$, the {\it induced subgraph}
of $G$ on $W$ is the graph $G_W = (W,E(G_W))$ where 
$E(G_W) = \{e \in E(G) \mid e \subseteq W \}$.  Given
an $x \in V(G)$, the set of {\it neighbours} of $x$ is the 
set $N(x) = \{y ~|~ \{x,y\} \in E(G)\}$.

A set of  vertices $W \subseteq V$ is an {\it independent set} if for all
$e \in E$, $e \not\subseteq W$.  An independent set is a {\it maximal
independent set} if it is maximal with respect to inclusion.  We let
$\alpha(G)$ denote the size of the largest maximal independent set.
Using the independent sets, we can build a simplicial complex.  In particular,
the {\it independence complex} of $G$ is the simplicial complex:
\[{\rm Ind}(G) = \{ W \subseteq V ~|~ W \mbox{ is an independent set}\}.\]
Note that $\alpha(G)$ is the cardinality of the largest element in ${\rm Ind}(G)$.

A set of vertices $W \subseteq V$ is a {\it vertex cover} if for
all $e \in E$, $e \cap W \neq \emptyset$.  A vertex cover is a
{\it minimal vertex cover} if it is minimal with respect to inclusion.
We let $\beta(G)$ denote the size of the smallest minimal vertex cover.
There is duality between independent sets and vertex covers;  specifically,
$W \subseteq V$ is an independent set if and only if $V \setminus W$ is
a vertex cover.  Consequently
\begin{equation}\label{alphabeta}
 \alpha(G) + \beta(G) = n.
\end{equation}

A set of edges $\{e_1,\ldots,e_s\} \subseteq E$ is said to be a
{\it matching} if none of the edges share a common vertex.   
We let $\alpha'(G)$ denote the size of the maximum matching in $G$.
We then always have the following inequality:
\begin{equation}\label{inequality}
\alpha'(G) \leq \beta(G).
\end{equation}
Indeed, for any matching $\{e_1,\ldots,e_s\} \subseteq E$, 
any minimal vertex cover must contain at least one vertex from each $e_i$

Finally, we will require the following bound on the regularity of $R/I(G)$.

\begin{theorem}[{\cite[Theorem 6.7]{HVT}}]\label{hvtbound}
For any finite simple graph $G$, 
${\rm reg}\left(R/I(G)\right) \leq \alpha'(G).$
\end{theorem}

\section{Main Theorem}

In this section we will prove our main theorem:

\begin{theorem}\label{Main1}
Let $r, d \geq 1$ be integers.    
Then there is a finite simple graph
$G$ with $r = {\rm reg}\left(R/I(G)\right)$ and $d = \deg h_{R/I(G)}(t)$.
\end{theorem}

In order to show this theorem, we will prepare some lemmata. 

\begin{lemma}[{\cite[Lemma 3.2]{HT}}]
Let $R_{1} = k[x_{1}, \ldots, x_{n'}]$ and 
$R_{2} = k[x_{n' + 1}, \ldots, x_{n}]$ be polynomial rings 
over a field $k$.  Let $I_{1}$, respectively $I_2$,
 be a nonzero homogeneous ideal of 
$R_{1}$, respectively $R_{2}$. 
We write $R$ for $R_{1} \otimes_{k} R_{2} = k[x_{1}, \ldots, x_{n}]$ and 
regard $I_{1} + I_{2}$ as a homogeneous ideal of $R$. 
Then 
\begin{eqnarray*}
{\rm reg}(R/I_{1} + I_{2}) & =& {\rm reg}(R_{1}/I_{1}) + {\rm reg}(R_{2}/I_{2}), 
~~\mbox{and} \\
H_{R/I_{1} + I_{2}} (t) & =& H_{R_{1}/I_{1}} (t) \cdot H_{R_{2}/I_{2}} (t). 
\end{eqnarray*}
\end{lemma}

By virtue of this lemma, one has:

\begin{lemma}\label{Lemma}
Let $G$ be a simple graph, and let 
$G_{1}, \ldots G_{\ell}$ be the connected components of $G$. 
Then
\begin{eqnarray*}
{\rm reg}\left(R/I(G)\right) & = &\sum_{i = 1}^{\ell} {\rm reg}\left(R_{i}/I(G_{i})\right),  
~~\mbox{and}~~ \deg h_{R/I(G)}(t)  =  \sum_{i = 1}^{\ell} \deg h_{R_{i}/I(G_{i})}(t), 
\end{eqnarray*}
where $R_{i} = k\left[x_{j} \mid j \in V(G_{i})\right]$ for $i=1,\ldots,\ell$, 
and 
$R = R_1 \otimes_k \cdots \otimes_k R_\ell$.
\end{lemma}

\begin{remark}\label{remark1}
By Lemma \ref{Lemma}, if $G$ is graph with
${\rm reg}\left(R/I(G)\right) = r$ and $\deg h_{R/I(G)}(t) = d$, then the graph
$G'$ which is the disjoint union of $G$ and a single edge on two
new vertices $\{z_1,z_2\}$ has ${\rm reg}(R'/I(G')) = r+1$ and $\deg h_{R'/I(G')}(t) = d+1$ where $R' = R \otimes_k k[z_1,z_2]$.
To prove Theorem
\ref{Main1} we need to show that for each $r \geq 1$, there
exists a graph $G$ with 
${\rm reg}(R/I(G)) = r$ and $\deg h_{R/I(G)}(t) =1$, and for each
$d \geq 1$, there is a 
graph $G$ with ${\rm reg}(R/I(G)) = 1$ and $\deg h_{R/I(G)}(t) =d$.  
We now work towards this goal.
\end{remark}

\begin{example}\label{CompleteBipartite}
Let $d \geq 1$ be a positive integer and let $K_{d, d}$ be the 
complete bipartite graph, i.e., the graph with $V(K_{d,d}) = \{x_1,\ldots,x_d,
y_1,\ldots,y_d\}$ and $E(K_{d,d}) = \{x_iy_j ~|~ 1 \leq i, j\leq d\}$. 
By virtue of Fr\"oberg's Theorem \cite[Theorem 1]{F}, 
one has ${\rm reg}(R/I(K_{d, d})) = 1$. 
In addition, the Hilbert series of $R/I(K_{d,d})$ can be computed
from the graded minimal free resolution (e.g., see \cite[Theorem 5.2.4]{J});  
in particular:
\[ H_{R/I(K_{d, d})}(t) = \frac{-(1 - t)^{d} + 2}{(1 - t)^{d}}.\] 
Hence $\deg h_{R/I(K_{d, d})}(t) = d$. 
\end{example}
We now require the following graph construction.
Let $G$ be a simple graph on $V(G) = \{x_1,\ldots,x_n\}$. 
For $S \subset V(G)$, the graph $G^{S}$ is defined by  
\begin{itemize}
	\item $V(G^{S}) = V(G) \cup \{x_{n + 1}\}$, where $x_{n + 1}$ is a new vertex; and 
	\item $E(G^{S}) = E(G) \cup \{ \{x_i, x_{n  + 1}\} \mid x_i \in S \}$. 
\end{itemize}

\begin{lemma}\label{LemmaA}
Let $G$ be a graph and let $S \subset V(G)$. 
Assume that 
\begin{itemize}
	\item $\dim R/I(G) \geq 2$ and  $h_{R/I(G)} (t) = 1 + h_{1}t + h_{2}t^{2}$; 
	\item ${\rm reg}\left(R/I(G)\right) \geq 2$; 
	\item $|S| = |V(G)| - \dim R/I(G) + 2$; and
	\item For any $u \in V(G) \setminus S$, there exists $u' \in S$ such that $\{u, u'\} \in E(G)$.   
\end{itemize}
Then 
\[
H_{R'/I(G^{S})}(t) = \frac{1 + (h_{1} + 1)t + (h_{2} - 1)t^{2}}{(1 - t)^{\dim R/I(G)}}
~~\mbox{and}~~ {\rm reg}\left(R'/I(G^{S})\right) = r,
\] 
where $R' = R \otimes_{k} k[x_{n + 1}]$. 
\end{lemma}
\begin{proof}
By the assumptions and the definition of $G^{S}$, we have 
$I(G^{S}) + (x_{n + 1}) = (x_{n + 1}) + I(G)$, 
and $I(G^{S}):(x_{n + 1}) = (x_{i} \mid x_i \in S)$. 
Hence $R'/(I(G^{S}) + (x_{n + 1})) \cong R/I(G)$ and 
$R'/(I(G^{S}) : (x_{n + 1})) \cong k[x_{i} \mid x_i \not\in S] \otimes_{k} k[x_{n + 1}]$. 
Thus, by the additivity of Hilbert series on the short exact sequence 
\[
0 \to \left(R'/(I(G^{S}) : (x_{n + 1}))\right)(-1) \xrightarrow{\ \times x_{n + 1} \ } R'/I(G^{S}) \to R'/(I(G^{S}) + (x_{n + 1})) \to 0, 
\]
we have 
\begin{eqnarray*}
H_{R'/I(G^{S})} (t) &=& H_{R'/(I(G^{S}) + (x_{n + 1}))} (t) + t \cdot 
H_{R'/(I(G^{S}) : (x_{n + 1}))} (t) \\
&=& H_{R/I(G)} (t) + \frac{t}{(1 - t)^{|V(G)| - |S| + 1}} \\
&=& \frac{1 + h_{1}t + h_{2}t^{2}}{(1 - t)^{\dim R/I(G)}} + \frac{t}{(1 - t)^{\dim R/I(G)-1}} \\ 
&=& \frac{1 + (h_{1} + 1)t + (h_{2} - 1)t^{2}}{(1 - t)^{\dim R/I(G)}}. 
\end{eqnarray*} 
Furthermore, we have ${\rm reg}\left(R'/I(G^{S})\right) = r$ 
by virtue of \cite[Lemma 2.10]{DHS}. 
\end{proof}

\begin{example}\label{ribbon}
Let $G$ be the two disjoint edges $\{x_1,x_2\}$ and $\{x_3,x_4\}$
 and $S = V(G)$.   Then $G^S = 
G_{\rm ribbon}$ where $G_{\rm ribbon}$ is the following graph: 

\bigskip

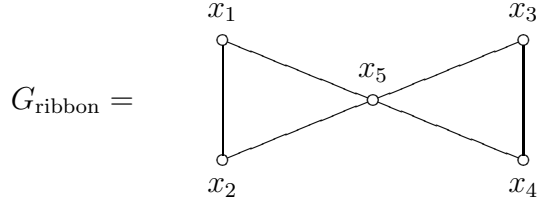
\begin{xy}
	\ar@{} (0,0);(50,0) *{\text{$G_{\rm ribbon} =$}};
	\ar@{} (0,0);(90, 0)  *++!D{x_5} *\cir<2pt>{} = "A";
	\ar@{-} "A";(70, 8)   *++!D{x_1} *\cir<2pt>{} = "B";
	\ar@{-} "A";(70, -8)  *++!U{x_2} *\cir<2pt>{} = "C";
	\ar@{-} "A";(110, 8)  *++!D{x_3}  *\cir<2pt>{} = "D";
	\ar@{-} "A";(110, -8) *++!U{x_4} *\cir<2pt>{} = "E";
	\ar@{-} "B";"C";
	\ar@{-} "D";"E";
\end{xy}

\bigskip 

\noindent
Since $I(G) = (x_1x_2,x_3x_4)$ is a complete intersection, we have 
$\displaystyle H_{R/I(G)} (t) = \frac{1 + 2t + t^{2}}{(1 - t)^{2}}$ and 
${\rm reg} (R/I(G)) = 2$. 
Hence, by applying Lemma \ref{LemmaA}, one has 
\[ H_{R'/I(G_{\rm ribbon})}(t) = \frac{1 + 3t}{(1 - t)^{2}} ~~\mbox{and}~~
 {\rm reg}(R'/I(G_{\rm ribbon})) = 2.\]
So $\deg h_{R'/I(G_{\rm ribbon})}(t) = 1$. 
\end{example}

\begin{example}\label{r=3}
Let $G_{0}$ be the union of $G_{\rm ribbon}$ and a disjoint edge $\{x_6,x_7\}$:  

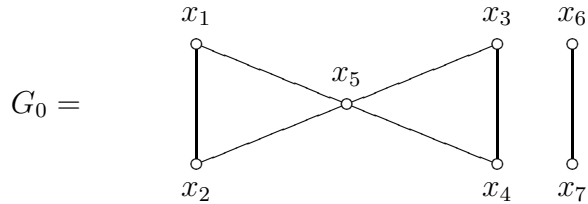
\bigskip

\begin{xy}
	\ar@{} (0,0);(40,0) *{\text{$G_{0} = $}};
	\ar@{} (0,0);(80, 0) *++!D{x_5} *\cir<2pt>{} = "A";
	\ar@{-} "A";(60, 8) *++!D{x_1} *\cir<2pt>{} = "B";
	\ar@{-} "A";(60, -8) *++!U{x_2} *\cir<2pt>{} = "C";
	\ar@{-} "A";(100, 8) *++!D{x_3} *\cir<2pt>{} = "D";
	\ar@{-} "A";(100, -8) *++!U{x_4} *\cir<2pt>{} = "E";
	\ar@{-} "B";"C";
	\ar@{-} "D";"E";
	\ar@{} (0,0);(110, 8) *++!D{x_6} *\cir<2pt>{} = "F11";
	\ar@{-} "F11";(110, -8) *++!U{x_7} *\cir<2pt>{} = "F12";
\end{xy}

\bigskip

\noindent Then $\displaystyle H_{R/I(G_{0})}(t) = \frac{1 + 3t}{(1 - t)^{2}} \cdot \frac{1 + t}{1 - t} = \frac{1 + 4t + 3t^{2}}{(1 - t)^{3}}$ and ${\rm reg}(R/I(G_{0})) = 2 + 1 = 3$ by virtue of Lemma \ref{Lemma} and 
Example \ref{ribbon}. 
Now we set $S_{i} = V(G_{i}) \setminus \{x_7\}$ and 
$G_{i + 1} = G^{S_{i}}_{i}$ for $i = 0, 1, 2$. 
Then, by using Lemma \ref{LemmaA} repeatedly, one has 
\[
H_{R'/I(G_{3})}(t) = \frac{1 + 7t}{(1 - t)^{3}}~~\mbox{and}~~ 
{\rm reg}\left(R'/I(G_{3})\right) = 3, 
\]
where $R' = k[x_{1}, \ldots, x_{10}]$ 
and $G_{3}$ is the following graph:

\bigskip

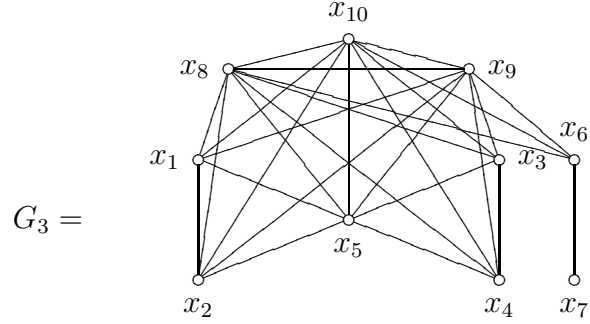
\begin{xy}
	\ar@{} (0,0);(40,0) *{\text{$G_{3} = $}};
	\ar@{} (0,0);(80, 0) *++!U{x_5} *\cir<2pt>{} = "A";
	\ar@{-} "A";(60, 8) *++!R{x_1} *\cir<2pt>{} = "B";
	\ar@{-} "A";(60, -8) *++!U{x_2} *\cir<2pt>{} = "C";
	\ar@{-} "A";(100, 8) *++!L{x_3} *\cir<2pt>{} = "D";
	\ar@{-} "A";(100, -8) *++!U{x_4} *\cir<2pt>{} = "E";
	\ar@{-} "B";"C";
	\ar@{-} "D";"E";
	\ar@{} (0,0);(110, 8) *++!D{x_6} *\cir<2pt>{} = "F11";
	\ar@{-} "F11";(110, -8) *++!U{x_7} *\cir<2pt>{} = "F12";
	\ar@{} (0,0);(64, 20) *++!R{x_8} *\cir<2pt>{} = "G";
	\ar@{} (0,0);(96, 20) *++!L{x_9} *\cir<2pt>{} = "H";
	\ar@{} (0,0);(80, 24) *++!D{x_{10}} *\cir<2pt>{} = "I";
	\ar@{-} "A";"G";
	\ar@{-} "B";"G";
	\ar@{-} "C";"G";
	\ar@{-} "D";"G";
	\ar@{-} "E";"G";
	\ar@{-} "F11";"G";
	\ar@{-} "A";"H";
	\ar@{-} "B";"H";
	\ar@{-} "C";"H";
	\ar@{-} "D";"H";
	\ar@{-} "E";"H";
	\ar@{-} "F11";"H";
	\ar@{-} "G";"H";
	\ar@{-} "A";"I";
	\ar@{-} "B";"I";
	\ar@{-} "C";"I";
	\ar@{-} "D";"I";
	\ar@{-} "E";"I";
	\ar@{-} "F11";"I";
	\ar@{-} "G";"I";
	\ar@{-} "H";"I";
\end{xy}

\bigskip

\end{example}

Lemma \ref{LemmaA} says that, given $r \geq 2$,  we can construct 
a graph $G'$ for which 
$\deg h_{R/I(G')}(t) = 1$ and ${\rm reg}(R/I(G')) = r$ 
from a graph $G$ for which 
$\deg h_{R/I(G)}(t) = 2$ and ${\rm reg}\left(R/I(G)\right) = r$, provided
the hypotheses of Lemma \ref{LemmaA} are met.  
We use this idea in the next lemma.
\begin{lemma} \label{LemmaB}
Given an integer $r \geq 3$, we put 
\[
Y_{r} = \{ y_{1, 1}, y_{2, 1} \ldots, y_{r-2, 1}, y_{1, 2}, y_{2, 2}, \ldots, y_{r - 2, 2} \}, 
\]
\[
Z_{r} = \bigcup_{i = 1}^{r - 2} \left\{ z^{(i)}_{1}, z^{(i)}_{2}, \ldots, z^{(i)}_{2^{i + 1} - 1} \right\}
\]
and
\[
X = \{x_{1}, x_{2}, x_{3}, x_{4}, x_{5}\}. 
\]
Let $G^{(r)}$ be the graph on $X \cup Y_{r} \cup Z_{r}$ such that 
\begin{itemize}
	\item the induced subgraph $G^{(r)}_{X, Y_{r}}$ is the following: \\
	
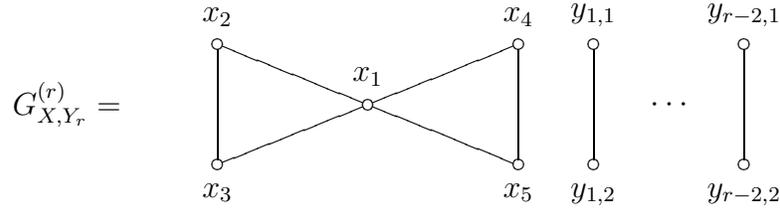
		\begin{xy}
			\ar@{} (0,0);(40,0) *{\text{$G^{(r)}_{X, Y_{r}} =$}};
			\ar@{} (0,0);(80, 0) *++!D{x_{1}} *\cir<2pt>{} = "A";
			\ar@{-} "A";(60, 8) *++!D{x_{2}} *\cir<2pt>{} = "B";
			\ar@{-} "A";(60, -8) *++!U{x_{3}} *\cir<2pt>{} = "C";
			\ar@{-} "A";(100, 8) *++!D{x_{4}} *\cir<2pt>{} = "D";
			\ar@{-} "A";(100, -8) *++!U{x_{5}} *\cir<2pt>{} = "E";
			\ar@{-} "B";"C";
			\ar@{-} "D";"E";
			\ar@{} (0,0);(110, 8) *++!D{y_{1, 1}} *\cir<2pt>{} = "F11";
			\ar@{-} "F11";(110, -8) *++!U{y_{1, 2}} *\cir<2pt>{} = "F12";
			\ar@{} (0,0);(120,0) *{\cdots};
			\ar@{} (0,0);(130, 8) *++!D{y_{r-2, 1}} *\cir<2pt>{} = "F21";
			\ar@{-} "F21";(130, -8) *++!U{y_{r-2, 2}} *\cir<2pt>{} = "F22";
		\end{xy}
	\item the induced subgraph $G^{(r)}_{Z_{r}}$ is a complete graph, i.e.,
all vertices are adjacent; and
	\item for all $1 \leq i \leq r - 2$ and for all $1 \leq j \leq 2^{i + 1} - 1$, 
	\[
	N_{G}\left(z^{(i)}_{j}\right) = X \cup \{y_{1, 1}, y_{2, 1}, \ldots, y_{r-2, 1}\} \cup Z_{r} \setminus \left\{ z^{(i)}_{j} \right\}. 
	\]
\end{itemize}
Let $R^{(r)} = k\left[ \{X \cup Y_{r} \cup Z_{r}\} \right]$ be the polynomial ring over $k$ 
whose variables equal to $X \cup Y_{r} \cup Z_{r}$. 
Then 
\begin{enumerate}
	\item $\displaystyle H_{R^{(r)} / I(G^{(r)})} (t) = \frac{1 + (2^{r} - 1)t}{(1 - t)^{r}}$, that is, $\deg h_{R^{(r)}/I(G^{(r)})}(t) = 1$, and
	\item ${\rm reg} \left(R^{(r)} / I(G^{(r)})\right) = r$. 
\end{enumerate}
\end{lemma}

\begin{proof}
We prove this lemma by induction on $r \geq 3$. 
The graph of  Example \ref{r=3} is $G^{(3)}$; we showed 
that $\displaystyle H_{R^{(3)} / I(G^{(3)})} (t) = \frac{1 + 7t}{(1 - t)^{3}}$ 
and ${\rm reg} \left(R^{(3)} / I(G^{(3)})\right) = 3$. 

Assume $r > 3$. 
Let $G'$ be the union of $G^{(r - 1)}$ and a disjoint edge $\{ y_{r-2, 1}, y_{r-2, 2} \}$. 
Let $R' = R^{(r - 1)} \otimes_{k} k[y_{r-2, 1}, y_{r-2, 2}]$. 
Then
\begin{eqnarray*}
H_{R'/I(G')}(t) &=& H_{R^{r-1}/I(G^{(r-1)})}(t) \cdot \frac{1 + t}{1 - t} = \frac{1 + (2^{r-1} - 1)t}{(1 - t)^{r-1}} \cdot \frac{1 + t}{1 - t} \\
&=& \frac{1 + 2^{r-1}t + (2^{r-1} - 1)t^{2}}{(1 - t)^{r}}
\end{eqnarray*}
and 
\[
{\rm reg}(R'/I(G')) = {\rm reg}\left(R^{(r-1)}/I(G^{(r-1)})\right) + 1 = r - 1 + 1 = r
\]
by the induction hypothesis and Lemma \ref{LemmaA}. 

Let $S_{0} = X \cup \{ y_{1, 1}, y_{1, 2}, \ldots, y_{r - 2, 1} \} \cup Z_{r - 1}$.  
Then $|S_{0}| = r + 3 + |Z_{r - 1}|$ and
\begin{eqnarray*}
|V(G')| - \dim R'/I(G') + 2 &=& |X| + |Y_{r-1}| +  |Z_{r - 1}| + 2 - r + 2 \\
&=& 5 + 2(r - 3) + |Z_{r - 1}| + 4 - r \\
&=& r + 3 + |Z_{r - 1}|. 
\end{eqnarray*}
Hence, by virtue of Lemma \ref{LemmaA}, one has 
\[
H_{R_{0}/I(G_{0})}(t) = \frac{1 + (2^{r-1} + 1)t + (2^{r-1} - 2)t^{2}}{(1 - t)^{r}}
~~\mbox{and}~~{\rm reg}(R_{0}/I(G_{0})) = r, 
\]
where $R_{0} = R' \otimes_{k} k\left[z^{(r-2)}_{1}\right]$, $G_{0} = (G')^{S_{0}}$, and $V(G_{0}) = V(G') \cup \left\{z^{(r-2)}_{1}\right\}$. 

Now, for each $1 \leq j \leq 2^{r-1} - 2$, we define $R_{j}$, $S_{j}$ and $G_{j}$ inductively: 
\begin{itemize}
	\item $R_{j} = R_{j-1} \otimes_{k} k\left[z^{(r-2)}_{j+1}\right]$;
	\item $S_{j} = S_{j-1} \cup \left\{ z^{(r-2)}_{j+1} \right\}$; and
	\item $G_{j} = G^{S_{j}}_{j - 1}$. 
\end{itemize}
Then $R_{2^{r-1} - 2} = R^{(r)}$, $G_{2^{r-1} - 2} = G^{(r)}$, and one has 
\[
H_{R^{(r)}/I\left(G^{(r)}\right)}(t) = \frac{1 + (2^{r} - 1)t}{(1 - t)^{r}} 
~~\mbox{and}~~ {\rm reg}\left(R^{(r)}/I(G^{(r)})\right) = r
\]
by using Lemma \ref{LemmaA} repeatedly. 
\end{proof}

We are now in a position to finish the proof of Theorem \ref{Main1}. 

\begin{proof}[Proof (of Theorem \ref{Main1}).]
We discuss each of the following three cases. 

{\bf Case 1.} Suppose that $1 \leq r \leq d$. 
Let $G$ be the union of $K_{d - r + 1, d - r + 1}$ and $(r-1)$ disjoint edges. 
By virtue of Lemma \ref{Lemma} and Example \ref{CompleteBipartite}, one has
\[
{\rm reg}\left(R/I(G)\right) = 1 + (r - 1) = r~~\mbox{and}~~
\deg h_{R/I(G)}(t) = (d - r + 1) + (r - 1) = d. 
\]

{\bf Case 2.} Suppose  that $r, d \geq 1$ are integers with $r - d = 1$. 
Let $G$ be the union of $G_{\rm ribbon}$ and 
$(r -  2)$ disjoint edges. 
By virtue of Lemma \ref{Lemma} and Example \ref{ribbon}, one has
\[
{\rm reg}\left(R/I(G)\right) = 2 + (r  - 2) = r,  
~~\mbox{and}~~
\deg h_{R/I(G)}(t) = 1 + (r - 2) = r - 1 = d. 
\]

{\bf Case 3.} Suppose that  $r, d \geq 1$ are integers with $r - d > 1$.
Let $G$ be the union of the graph $G^{(r - d + 1)}$ of Lemma \ref{LemmaB} and 
$(d-1)$ disjoint edges. 
By virtue of Lemma \ref{Lemma} and \ref{LemmaB}, one has
\[
{\rm reg}\left(R/I(G)\right) = (r - d + 1) + (d - 1) = r,  
~~\mbox{and}~~
\deg h_{R/I(G)}(t) = 1 + (d - 1) = d. 
\]
\end{proof}

\section{Comparing the regularity and $h$-polynomial for fixed $n$}

Theorem \ref{Main1} shows that for all $(r,d) \in \mathbb{N}^2_{\geq 1}$, 
there exists a finite simple graph $G$ with 
$\left({\rm reg}\left(R/I(G)\right),
\deg h_{R/I(G)}(t)\right) = (r,d)$.  However, if we fix $n = |V(G)|$, then 
the regularity of $R/I(G)$ and the degree of the $h$-polynomial
must also satisfy the following inequality:

\begin{theorem} \label{maintheorem2}
 Let $G$ be a finite simple graph on $n$ vertices.
Then 
\[
\deg h_{R/I(G)}(t) + {\rm reg}\left(R/I(G)\right) \leq n. 
\]

\end{theorem}

\begin{proof}
Via the Stanley-Reisner correspondence, the edge ideal $I(G)$ is
associated to the independence complex ${\rm Ind}(G)$.
The Hilbert series of $R/I(G)$ can then be expressed as
\[ H_{R/I(G)}(t) = \sum_{i=0}^d \frac{f_{i-1}t^i}{(1-t)^i}\]
(see \cite[Theorem 6.2.1]{hhGTM})
where $f_{j-1}$ is the number of independent sets of cardinality $j$ in $G$ (in
other words, this in the number of faces of ${\rm Ind}(G)$ of dimension $j-1$).
In particular, $d = \alpha(G)$, the size of the largest independent set.
It follows that $\deg h_{R/I(G)}(t) \leq \alpha(G)$.
By combining Theorem \ref{hvtbound} and the inequality \eqref{inequality},
we have the bound 
${\rm reg}\left(R/I(G)\right) \leq \alpha'(G) \leq \beta(G)$.  
Thus
\[\deg h_{R/I(G)}(t) + {\rm reg}\left(R/I(G)\right) \leq \alpha(G) + \beta(G) = n,\]
as desired, where the last equality is \eqref{alphabeta}.
\end{proof}

\begin{remark}
For an alternative proof, one can use \cite[Corollary 4.3]{R}
to show that $\deg h_{R/I(G)}(t) \leq (n - \beta(G))$.
\end{remark}

\begin{example}
The upper bound of Theorem \ref{maintheorem2} is sharp. 
In fact, we can give two families of graphs such that the equality $\deg h_{R/I(G)}(t) + {\rm reg}\left(R/I(G)\right) = n$ holds. 
For the first family, let $n = 2m$ and let $G$ be the union of $m$ 
disjoint edges. Then $\deg h_{R/I(G)}(t) = {\rm reg}\left(R/I(G)\right) = m$. 
For the second family, let $G = K_{1, n - 1}$ be the star graph. 
Then $\deg h_{R/I(G)}(t) = n - 1$ and ${\rm reg}\left(R/I(G)\right) = 1$. 
\end{example}

\begin{remark} We end with an observation based upon our computer
experiments.  For any graph $G$ with at least one edge, we have
${\rm reg}(R/I(G)) \geq 1$ and ${\rm deg}(R/I(G)) \geq 1$.
If we fix an $n = |V(G)|$, it is natural to ask if we can describe all
pairs $(r,d) \in \mathbb{N}^2_{\geq 1}$ for which there is 
a graph $G$ on $n$ vertices with $r = {\rm reg}\left(R/I(G)\right)$
and $d = \deg h_{R/I(G)}(t)$.   Theorem \ref{maintheorem2} 
implies that $r+d \leq n$.  Furthermore,
note that $\alpha'(G) \leq \lfloor \frac{n}{2} \rfloor$,
so we must also have $r \leq \lfloor \frac{n}{2} \rfloor$
by Theorem \ref{hvtbound}. 

However, these inequalities are not enough to desribe all the
pairs $(r,d)$ that may be realizable.   For example, when $n=9$,
we computed  $\left({\rm reg}\left(R/I(G)\right),\deg h_{R/I(G)}(t) \right)$
for all 274668 graphs on nine vertices.  We observed that for all such
$G$,
\[\left({\rm reg}\left(R/I(G)\right),\deg h_{R/I(G)}(t) \right)
\not\in \{(3,1),(4,1),(4,2)\}\]
even though these tuples
 satisfy the inequalities $r+d \leq 9$ and $r \leq 4$.
A similar phenomenon was observed for other $n$, thus suggesting
the existence of another bound relating 
${\rm reg}\left(R/I(G)\right)$ and $\deg h_{R/I(G)}(t)$ for a fixed $n$
\end{remark}

\bibliographystyle{plain}

\end{document}